\documentclass{article}
\usepackage{amsmath,amssymb,amsthm}
\newtheorem{thm}{Theorem}[section] 
\newtheorem{prop}[thm]{Proposition} 
\newtheorem{defi}[thm]{Definition}

\newtheorem{lem}[thm]{Lemma}

\begin{document}

\title{\LARGE The number of relatively $r$-prime $k$-tuple integers}
\author{\Large Wataru Takeda}
\date{\normalsize Department of Mathematics,\\ Kyoto University, \\Kitashirakawa Oiwake-cho, Sakyo-ku, Kyoto 606-8502,
Japan}
\maketitle
\begin{abstract}
For a fixed integer $r\ge1$, we say $k$-tuple integers $(x_1,\ldots,x_k)$ are relatively $r$-prime if there exists no prime $p$ such that all $k$ integers is multiple of $p^r$. Benkoski proved that the number of relatively $r$-prime $k$-tuple integers in $[1,x]^k$ is $x^k/\zeta(rk)+$(Error term) \cite{Be76}. We showed that the exact order of error term is $x^{k-1}$ for $rk\ge3$ and $k\not=1$. 
\end{abstract}

\setcounter{section}{-1}
\section{Introduction}
From 1800's many results about the distribution of special lattice points were shown.  F. Mertens proved that the density of the set of coprime pairs of integers is $1/\zeta(2)$ in 1874 \cite{Me74}. And this result was extended to $k$-tuple integers by D. N. Lehmer \cite{Le00}. On the other hand, Gegenbauer proved that the probability that an integer is $r$-free is $1/\zeta(r)$ in 1885 \cite{Ge85}. As a generalization of these result, S. J. Benkoski proved that the density of the set of relatively $r$-prime $k$-tuple integers is $1/\zeta(rk)$ in 1976 \cite{Be76}. 

In \cite{Ta16}, we computed the number of coprime $k$-tuple integers in $[-x,x]^k$ and the exact order of magnitude of its error term is $x^{k-1}$ for all $k\ge3$. We will generalize this result to relatively $r$-prime $k$-tuple integers. For fixed $r\ge1$ let $(x_1,\ldots,x_k)_r$ be integer $n$ such that $n$ is the greatest common factor of the form $n^r$ for integer $n\ge1$. When $r=1$,  $(x_1,\ldots,x_k)_1$ means the great common divisor of $x_1,\ldots,x_k$ i.e. $\gcd(x_1,\ldots,x_k)$. And let $V_k^r(x)$ denote the number of $k$-tuple integers $(x_1,\ldots,x_k)$ such that $(x_1,\ldots,x_k)_r=1$ and $|x_i|\le x$ for all $i=1,\ldots,k$. When $r=1$, $V_k^1(x)$ means the number of visible lattice points in $[-x,x]^k$ and $k=1$ a half of $V_1^r(x)$ means the number of $r$-free positive integers $\le x$. And we let $E_k^r(x)$ denote the error term, i.e. $E_k^r(x)=V_k^r(x)-(2x)^k/\zeta(rk)$.

In this paper, we compute $V_k^r(x)$ by following the ways of \cite{Ta16}, so we get a generating function of $V_k^r(x)$ for a fixed positive integer $r$ and the exact order of $E_k^r(x)$ is $x^{k-1}$ for $rk\ge3$ and $k\not=1$. 

\section{Benkoski's result}
To consider the exact order of $E_k^r(x)$, we use S. J. Benkoski result. He considered that the number of $r$-prime $k$-tuple integers $(x_1,\ldots,x_k)$ such that $1\le|x_i|\le x$ for all $i=1,\ldots,k$ by using a general Jordan totient function $J_k^r(n)$ (Theorem 3, 4 and 5 in \cite{Be76}). The general Jordan totient function is defined as follows.

\begin{defi}[Definition of Chapter 2. \cite{Be76}] 
Let $r\ge1$, $k\ge1$ are integers. 
\begin{align*}
J_k^r(n):=&|\{(x_1,...,x_k)\in\mathbb{Z}^k~|~(x_1\ldots,x_k,n)_r=1,1\le x_i\le n\ (1\le i\le k)\}|.\\
\intertext{For $k=0$ we define}
J_0^r(n):=&\left\{
\begin{array}{rl}
1&(n \text{ is $r$-free}),\\
0&(\text{otherwise}).
\end{array}
\right.
\end{align*}
\end{defi}
If $r=1$ then $J_k^1(n)$ is ordinaly Jordan totient function and if $r=k=1$ then $J_1^1(n)$ is the Euler totient function. We know that general Jordan totient function $J_k^r(n)$ has Dirichlet product and Euler product expansion,
\[J_k^r(n)=\sum_{d^r|n} \mu(d)\left(\frac n{d^r}\right)^k=n^k\prod_{\substack{p^r|n\\p:\text{prime}}}\left(1-\frac1{p^{rk}}\right).\]
When $n$ is $r$-free, the product is empty and assigned to be the value $1$. This formula is proved as well as an analogue statement of Euler totient functon. 

Benkoski considered only positive integers in his paper \cite{Be76}. But considering the sign of component of $(x_1\ldots,x_k)$, we obtain the following asymptotic formula from Benkoski's result.
\[V_k^r(x)=\frac{2^k}{\zeta(rk)}x^k+\left\{
\begin{array}{ll}
O(x\log x)& (r=1 \text{ and } k=2),\\
O(x^{1/r}) & (r\ge 2 \text{ and } k=1),\\
O(x^{k-1}) & (\text{otherwise}).
\end{array}
\right.\]

\section{The partial sums of the general Jordan totient function}
In this paper, we use the $\Omega$ simbol introduced by G.H. Hardy and J.E. Littlewood. This simbol is defined as follows:
\[f(x)=\Omega(g(x))\overset{def}{\Longleftrightarrow}\limsup_{x\rightarrow\infty}\left|\frac{f(x)}{g(x)}\right|>0.\]
If there exists a function $g(x)$ such that $f(x)=O(g(x))$ and $f(x)=\Omega(g(x))$ then the exact order of $f(x)$ is $g(x)$. 
In \cite{Ta16}, We calculated the exact order of $E_k^1(x)$ by using the following theorem.
\begin{thm}[Lemma 4.2. of \cite{Ta16}] For $r=1$ and $k\ge3$\[\sum_{n \le x}J_{k-1}^1(n)=\frac{x^k}{k\zeta(k)}+\Omega(x^{k-1})\]
\end{thm}
We follow this way to consider the exact order of $E_k^r(x)$ for $r\ge1$. So we consider an asymptotic formula for the partial sums of $J_{k-1}^r(n)$. Because the general Jordan totient function has Dirichlet product, by applying same argument of the Chapter 4 in \cite{Ta16}, we get following equation:
\[\sum_{n \le x}J_{k-1}^r(n)=\frac 1k \sum_{j=0}^{k-1} \binom kj B_j\sum_{d^r \le x}\mu(d)\left(\frac x{d^r}-\left\{\frac x{d^r}\right\}\right)^{k-j},\]
where $\{x\}$ is the fractional part of $x$ and $B_0,B_1,B_2,\ldots$ are Bernoulli numbers. (Note. We use the second Bernoulli number, i.e. $B_1=\displaystyle{\frac12}$.) To evaluate this sum, we will extend Lemma 4.1. of \cite{Ta16} to $r\ge2$ case.
\begin{lem}
\label{thm:2.2}
Let $\{x\}$ be the fractional part of $x$.
If $rk\ge2$, \[\sum_{d^r\le x}\mu(d)\left(\frac x{d^r}\right)^k\left\{\frac x{d^r}\right\}=\Omega(x^k).\]
\end{lem}

\begin{proof}
It suffices to show that $\displaystyle{\sum_{d^r\le x}\frac {\mu(d)}{d^{rk}}\left\{\frac x{d^{r}}\right\} \le M<0}$ for infinity many values of $x$ and some negative $M$.
The case of $r=1$ is proved in \cite{Ta16}, so it suffices to consider $r\ge2$.
\begin{align*}
\intertext{If $rk \ge 4$, let $x$ be a integer such that $x\equiv2^r-1\!\!\!\mod2^r$ and greater than or equal to $3^r$,}
\sum_{d^r\le x}\frac {\mu(d)}{d^{rk}}\left\{\frac x{d^{r}}\right\}&=\sum_{d\le \sqrt[r] x}\frac {\mu(d)}{d^{rk}}\left\{\frac x{d^{r}}\right\}\\
&=-\frac{2^r-1}{2^{r(k+1)}}+\sum_{3\le d\le \sqrt[r] x}\frac {\mu(d)}{d^{rk}}\left\{\frac x{d^{r}}\right\}.\\
\intertext{Since $\mu(d)=1.0.-1$ and $\displaystyle{\left\{\frac x{d^r}\right\}\le 1}$,}
\sum_{d^r\le x}\frac {\mu(d)}{d^{rk}}\left\{\frac x{d^{r}}\right\}&< -\frac 1{2^{rk}}+\frac1{2^{r(k+1)}}+\sum_{3\le d\le \sqrt[r]x}\frac 1{d^{rk}}\\
&<-\frac 1{2^{rk}}+\frac1{2^{r(k+1)}}+\zeta(rk)-1-\frac 1{2^{rk}}.\\
\intertext{When $rk\ge4$ we know that $\displaystyle{\zeta(rk)-1-\frac1{2^{rk}}<\frac1{2^{rk+1}}}$, so we get}
\sum_{d^r\le x}\frac {\mu(d)}{d^{rk}}\left\{\frac x{d^{r}}\right\}&<-\frac1{2^{rk+1}}+\frac1{2^{r(k+1)}}<0,
\intertext{since $r\ge2$.}
\intertext{So for $rk \ge 4$ the lemma follows.}
\end{align*}
Suppose that $(r,k)=(2,1)$ or $(r,k)=(3,1)$ and $\displaystyle{x=m^2\prod_{p\le 100}p^r}$, where the product is extended over all odd primes less than $100$ and $m$ isn't a multiple of $2$ and $p$.
\begin{align*}
\intertext{Then,}
\sum_{d^r\le x}\frac {\mu(d)}{d^r}\left\{\frac x{d^r}\right\}&= \sum_{d=1}^{100}\frac {\mu(d)}{d^r}\left\{\frac x{d^r}\right\}+\sum_{d=101}^{x^{1/r}}\frac {\mu(d)}{d^r}\left\{\frac x{d^r}\right\}.\\
\intertext{Since $\mu(d)=1.0.-1$ and $\displaystyle{\left\{\frac x{d^r}\right\}< 1}$,}
\sum_{d^r\le x}\frac {\mu(d)}{d^r}\left\{\frac x{d^r}\right\}&< -\frac1{2^{r+1}}+\sum_{d=3}^{100}\frac {\mu(d)}{d^r}\left\{\frac x{d^r}\right\}+\sum_{d=101}^{\infty}\frac 1{d^r}.
\end{align*}
Now we estimate how fast second sum grows.
When $r=2$ we obtain
\begin{align*}
\sum_{d=3}^{100}\frac {\mu(d)}{d^2}\left\{\frac x{d^2}\right\}&=\sum_{p=prime}^{47}\frac 1{(2p)^2}\frac14-\frac14\left(\frac1{30^2}+\frac1{42^2}+\frac1{66^2}+\frac1{78^2}+\frac1{70^2}\right)\\
&<\frac1{50}.
\intertext{On the other hand, when $r=3$,}
\sum_{d=3}^{100}\frac {\mu(d)}{d^3}\left\{\frac x{d^3}\right\}&=\sum_{p=prime}^{47}\frac 1{(2p)^3}\frac {\overline{p}}8-\frac18\left(\frac7{30^3}+\frac5{42^3}+\frac1{66^3}+\frac7{78^3}+\frac3{70^3}\right),\\
\intertext{where $\overline{p}\equiv p\!\!\!\mod8$ and  $0\le \overline{p}<8$. }
\sum_{d=3}^{100}\frac {\mu(d)}{d^3}\left\{\frac x{d^3}\right\}&<\frac25\times\frac1{10^{2}}.
\end{align*}
\begin{align*}
\intertext{From this result and we have $\displaystyle{\sum_{d=101}^{\infty}\frac 1{d^r}\le \frac1{100^{r-1}}}$, we find}
\sum_{d^r\le x}\frac {\mu(d)}{d^r}\left\{\frac x{d^r}\right\}&< -\frac1{2^{r+1}}+\frac25\times\frac1{10^{r-1}}+\frac1{100^{r-1}}\\
&<-\frac1{20},
\intertext{so for $(r,k)=(2,1)$ or $(r,k)=(3,1)$ the lemma follows.}
\intertext{This completes the proof of the lemma.}
\end{align*}
\end{proof}

As we remarked, the partial sums of $J_{k-1}^r(n)$ is equal to \[\frac 1k \sum_{j=0}^{k-1} \binom kj B_j\sum_{d^r \le x}\mu(d)\left(\frac x{d^r}-\left\{\frac x{d^r}\right\}\right)^{k-j}.\] We computed the order of the sum of $\displaystyle{\mu(d)\left(\frac x{d^r}\right)^{k-1}\left\{\frac x{d^r}\right\}}$ for all $rk\ge2$ in Lemma \ref{thm:2.2}. Next we will consider the order of principal term $\displaystyle{\sum_{d^r\le x}\mu(d)\frac{x^k}{d^{rk}}}$ of the partial sums of $J_k^r(n)$ in the Proposition \ref{prop}.
\begin{prop}
\label{prop}
For $rk\ge2$, $\displaystyle{\sum_{d^r\le x}\mu(d)\frac{x^k}{d^{rk}}=\frac{x^k}{\zeta(rk)}+O(x^{1/r})}.$
\end{prop}
\begin{proof}
We have
\[\sum_{d^r\le x}\mu(d)\frac{x^k}{d^{rk}}=\sum_{d\le x^{1/r}}\mu(d)\frac{x^k}{d^{rk}}.\]
We know that $\displaystyle{\sum_{d\le x}\frac{\mu(d)}{d^s}=\frac1{\zeta(s)}+O(x^{-s+1})}$ for $s>1$. (For the details for the proof of this result, one can see Theorem 11.7 of Apostol's book \cite{Ap76}).
\begin{align*}
\intertext{Use this asymptotic formula,}
\sum_{d^r\le x}\mu(d)\frac{x^k}{d^{rk}}&=x^k\left(\frac1{\zeta(rk)}+O\left((x^{1/r})^{-rk+1}\right)\right),\\
&=\frac{x^k}{\zeta(rk)}+O(x^{1/r}).
\end{align*}
This proposition holds.
\end{proof}
We note that for all $i$
\begin{align*}
\left|\sum_{d^r\le x}\frac {\mu(d)}{d^{rj}}\left\{\frac x{d^r}\right\}^i\right|&\le \sum_{d^r\le x}\frac 1{d^{rj}}=\left\{
\begin{array}{cl}
\zeta(rj)+O(x^{1/r-j})& (rj\ge2),\\
\log x+\gamma+o(1) & (rj=1),
\end{array}
\right.
\end{align*}
where $\gamma$ is Euler's constant, defined by the equation \[\gamma=\lim_{n\rightarrow\infty}\left(\sum_{k=1}^n\frac1k-\log n\right).\]
As we considered above, we get an order of all terms in the partial sums of $J_k^r(n)$. Using this result, we get the following theorem.
\begin{thm}
\label{thm:2.3}
For $rk\ge3$ and $k\not=1$, \[\sum_{n \le x}J_{k-1}^r(n)=\frac{x^k}{k\zeta(rk)}+\Omega(x^{k-1}).\]
\end{thm}

\begin{proof}
As already remarked, we know that 
\begin{align*}
\sum_{n \le x}J_{k-1}^r(n)&=\frac 1k \sum_{j=0}^{k-1} \binom kj B_j\sum_{d^r \le x}\mu(d)\left(\frac x{d^r}-\left\{\frac x{d^r}\right\}\right)^{k-j}.\\
\intertext{Using the binomial theorem and the order of summation,}
\sum_{n \le x}J_{k-1}^r(n)&=\frac 1k \sum_{j=0}^{k-1} \binom kj B_j\sum_{i=0}^{k-j}(-1)^i\sum_{d^r \le x}\mu(d)\left(\frac x{d^r}\right)^{k-j-i}\left\{\frac x{d^r}\right\}^i
\end{align*}
\begin{align*}
\intertext{Combining the remark before of this Theorem with Lemma \ref{thm:2.2} and Proposition \ref{prop}, we get}
\sum_{n \le x}J_{k-1}^r(n)&=x^k\sum_{d^r \le x}\frac {\mu(d)}{d^{rk}}+\Omega(x^{k-1})\\
&=\frac{x^k}{k\zeta(rk)}+\Omega(x^{k-1}).
\intertext{This proved the lemma.}
\end{align*}
\end{proof}

\section{Generating function of $V_k^r(x)$}
In this section, we consider a generating function of $V_k^r(x)$. The case of $r=1$ was considered in \cite{Ta16}. We will prove a generalisation of the case of $r=1$ by following the method of Theorem 3.1. of \cite{Ta16}.\begin{thm}
\label{thm:main}
Generating function of $V_k^r(x)$ is the following.
\begin{align*}
\sum_{k=0}^{\infty}\frac{u^k}{k!}V_k^r(x)&=\frac1{2u}(e^{(2X+1)u}-e^{(2X-1)u})\\
\intertext{and}
\sum_{k=0}^{\infty}u^{k+1}V_k^r(x)&=\frac12\log \frac{1-(2X-1)u}{1-(2X+1)u}, \\
\intertext{where $\displaystyle{k\sum_{n \le x}J_{k-1}^r(n)}$ is replaced by $X^k$ when $k\ge1$, and $X^0$ are assigned to be the value $0$.}
\end{align*}
\end{thm}
\begin{proof}
We can show these results as well as Theorem 3.1. of \cite{Ta16}. 
It suffices to show that \[V_k^r(x)=\frac1{2(k+1)}\{(2X+1)^{k+1}-(2X-1)^{k+1}\}.\]
After change of functions $J_k(n)$ into $J_k^r(n)$ in proof of Theorem 3.1. of \cite{Ta16}, we can compute $V_k^r(x)$ in same combinatorial way.
\begin{align*}
\intertext{Hence we obtain following equation}
V_k^r(x)&=\sum_{i=0}^{k-1}\binom ki 2^{k-i}\left(\sum_{n \le x}\sum_{j=0}^{k-i-1}(-1)^{k-i-1-j}\binom {k-i}jJ_j(n)\right).\\
\intertext{Applying the binomial theorem and changing the order of summation of it, we show}
V_k^r(x)&=\frac1{2(k+1)}\{(2X+1)^{k+1}-(2X-1)^{k+1}\}.
\intertext{This proves the theorem.}
\end{align*}
\end{proof}

\section{The exact order of magnitude of $E_k^r(x)$}
We showed that $V_k^r(x)$ is finite linear combination of $\displaystyle{\sum_{n \le x}J_{k-1}^r(n)}$ in last section. Combining this result with Theorem \ref{thm:2.3}, we get the exact order of magnitude of $E_k^r(x)$ as follows.

\begin{thm}
If $rk\ge3$ and $k\not=1$, \[E_k^r(x)=\Omega(x^{k-1}).\]
\end{thm}
\begin{proof}
We can prove this theorem easily from Theorem \ref{thm:main}.
\begin{align*}
\intertext{From Theorem \ref{thm:main}}
V_k^r(x)&=\frac1{2(k+1)}\{(2X+1)^{k+1}-(2X-1)^{k+1}\}\\
&=(2X)^{k}+O(X^{k-2}).\\
\intertext{Applying Theorem \ref{thm:2.3}, we find}
V_k^r(x)&=\frac{2^k}{\zeta(rk)}x^k+\Omega(x^{k-1}).\\
\intertext{Hence $E_k^r(x)=\Omega(x^{k-1})$ for $rk\ge3$ and $k\not=1$}
\end{align*}
\end{proof}
Combine Benkoski's result with this theorem, we prove that the exact order of magnitude of $E_k^r(x)$ is $x^{k-1}$, for all $rk\ge3$ and $k\not=1$.

\end{document}